\documentclass[10pt,twoside]{siamart1116}
\usepackage[english]{babel}
\usepackage{graphicx,epstopdf,epsfig}
\usepackage{amsfonts,epsfig,fancyhdr,graphics, hyperref,amsmath,amssymb}
\usepackage{enumerate}

\newcommand{\HE}{Name of Handling Editor}
\newcommand{\DoS}{Month/Day/Year}
\newcommand{\DoA}{Month/Day/Year}
\newcommand{\CA}{N. Thome}

\newcommand{\Names}{M.V. Hern\'andez, M. Lattanzi, and N. Thome}
\newcommand{\Title}{On diamond partial order, one-sided star partial orders, and 1MP-inverses}

\newtheorem{thm}{Theorem}[section]

\newtheorem{propo}[thm]{Proposition}
\newtheorem{rem}[thm]{Remark}

\newcommand{\A}{{\ensuremath{\cal{A}}}}

\newcommand{\C}{{\ensuremath{\mathbb{C}}}}
\newcommand{\Cmn}{{\ensuremath{\C^{m\times n}}}}
\newcommand{\Cnm}{{\ensuremath{\C^{n\times m}}}}
\newcommand{\Cnn}{{\ensuremath{\C^{n\times n}}}}
\newcommand{\Cmm}{{\ensuremath{\C^{m\times m}}}}

\newcommand{\Ra}{{\ensuremath{\cal R}}}

\numberwithin{equation}{section}
\usepackage[left=2.7cm,top=3.5cm,right=2.7cm,bottom=3.5cm]{geometry}

\usepackage[latin1]{inputenc}
\usepackage{cite}

\begin{document}

\bibliographystyle{plain}

\setcounter{page}{1}

\thispagestyle{empty}

 \title{On diamond partial order, one-sided star partial orders, and 1MP-inverses\thanks{Received
 by the editors on \DoS.
 Accepted for publication on \DoA. 
 Handling Editor: \HE. Corresponding Author: \CA}}
 
\author{
M.V.\ Hern\'andez\thanks{Universidad Nacional de La Pampa, FCEyN, Uruguay 151, 6300 Santa Rosa, La Pampa, Argentina (mvaleriahernandez@exactas.unlpam.edu.ar, mblatt@exactas.unlpam.edu.ar). Partially supported by Universidad Nacional de La Pampa, Facultad de Ingenier\'{i}a (grant Resol. Nro. 135/19).}
\and
M.B. \ Lattanzi\footnotemark[2]
\and 
N. \ Thome\thanks{Instituto Universitario de
Matem\'atica Multidisciplinar, Universitat Polit\`ecnica de
Val\`encia, 46022 Valencia, Spain (njthome@mat.upv.es). Partially supported by Ministerio de
Ciencia e Innovaci\'on of Spain (Grant Red Tem\'atica RED2022-134176-T) and by Universidad Nacional de La Pampa, Facultad de Ingenier\'{i}a (grant Resol. Nro. 135/19).}
}

\pagestyle{myheadings}
\markboth{\Names}{\Title}

\maketitle

\begin{abstract}
This paper provides some new characterizations of the diamond partial order for rec\-tan\-gu\-lar matrices by using properties of inner inverses, minus order, and SVD decompositions. In addition, the recently introduced 1MP generalized inverse and its dual are used to characterize the diamond partial order as a ${\cal G}$-based one. Finally, the one-sided star partial order is investigated by using 1MP- and MP1-inverses.
\end{abstract}



\section{Introduction and background}\label{Sec1}

Let $\Cmn$ be the set of $m \times n$ complex matrices. For $A\in\Cmn$, $A^{\ast},$ $A^{-1}$, $\text{rk}(A)$ and ${\Ra}(A)$ denote the conjugate transpose, the inverse ($m=n$), the rank, and the range space of $A$, respectively. As usual, $I_n$ stands for the $n \times n$ identity matrix and $0_{m\times n}$ denotes the $m \times n$ zero matrix. The subscripts will be omitted when no confusion is caused.

Let $A \in \Cmn$ and $X \in \Cnm$. It is said that  
\begin{enumerate}[(1)]
\item[{\rm (1)}]  $X$ is a $\{1\}$-inverse (or inner inverse) of $A$ when $AXA=A$,
\item[{\rm (2)}]  $X$ is a $\{2\}$-inverse (or outer inverse) of $A$ when $XAX=X$,
\item[{\rm (3)}]  $X$ is a $\{3\}$-inverse of $A$ when $(AX)^*=AX$, 
\item[{\rm (4)}]  $X$ is a $\{4\}$-inverse of $A$ when $(XA)^*=XA$.
\end{enumerate}

The set of matrices $X \in\Cnm$ satisfying the equation (1) is denoted by $\A\{1\}$ and one element of $\A\{1\}$ is denoted by $A^-$. 

 The unique matrix satisfying the equations (1)-(4) is called the Moore-Penrose inverse of $A \in \Cmn$, and is denoted by $A^\dagger$.

\begin{propo}\label{propA+0}
Let $A \in \Cmn$. Then, the Moore-Penrose inverse of $A$ satisfies the following properties:
\begin{enumerate}[(a)]
\item
$A^*=A^* A A^\dag = A^\dag A A^*$.
\item
$(AA^\dag)^\dag=AA^\dag$ and $(A^\dag A)^\dag =A^\dag A$. \label{A_proj_orto_MP}
\item \label{MP_distributiva} $(AB)^\dag = B^\dag A^\dag \mbox{ if and only if } A^* A B B^* \mbox{ is hermitian.}$ 
\end{enumerate}
\end{propo}
\begin{thm}{\rm \cite[Theorem 2.3.8]{MiBhMa}}\label{Amenos}
Let $A \in {\mathbb C}^{m \times n}$ and $A^- \in \A\{1\}$. Then: 
\begin{equation*} 
\begin{split}
\A\{1\}\,&=\,\{A^- + U - A^-AUAA^-: U \in \Cnm \text{ arbitrary}\}\\
&=\,\{A^- + (I - A^-A)V + W(I - AA^-): V,W \in \Cnm \text{ arbitrary}\}.
\end{split}
\end{equation*}
\end{thm}
For any matrix $A \in \Cmn$ with ${\rm rk}(A)=a>0$, a singular value decomposition (SVD, for short) \cite{BIG,CaMe,MiBhMa} is given by
\begin{equation}\label{SVDA}
A= U\left( \begin{array}{cc}
D_a&0\\
0&0
\end{array}
\right) V^*,
\end{equation} where $U\in\mathbb{C}^{m \times m}$ and $V\in\mathbb{C}^{n \times n}$ are unitary matrices and $D_a \in \mathbb C^{a \times a}$ is a positive definite diagonal matrix. In this case, it is well known that the general form for $\{1\}$-inverses of $A$ is given by 
\begin{equation}\label{unoinversa}
A^- = V \left( \begin{array}{cc}
D_a^{-1}&A_{12}\\
A_{21}&A_{22}
\end{array}
\right) U^*,
\end{equation}
partitioned according to the partition of $A$. In particular, the Moore-Penrose inverse of $A$ can be represented as
\begin{equation}\label{unoinversa}
A^\dag = V \left( \begin{array}{cc}
D_a^{-1}&0\\
0&0
\end{array}
\right) U^*.
\end{equation}

Let $A, B \in\Cmn$. It is known that the \textit{space preorder} (denoted by $\prec^s$) and the \textit{minus partial order} (denoted by $\leq^-$) are defined as:
\begin{enumerate}
\item $A\,\prec^s\,B$ iff ${\Ra}(A) \subseteq {\Ra}(B)$ and ${\Ra}(A^*) \subseteq {\Ra}(B^*)$;
\item $A\,\leq^- \,B$ iff there exists $A^- \in \A\{1\}$ such that $A^- A = A^- B$ and $A A^-=BA^-$.
\end{enumerate} 
The \textit{diamond partial order} for complex matrices, denoted by $\leq^\diamond$,    was introduced by J.K. Baksalary and J. Hauke in \cite{BH} in the following way. For two given matrices $A, B \in\Cmn$, 

\centerline{$A\,\leq^\diamond\,B \quad \text{iff} \quad A\,\prec^s\,B$ and $AB^*A= AA^*A$.}
Some properties of this partial order can be found in \cite{ArMa,BuMaMo,LePaTh}.

The following properties are known (see \cite{Ba,BH,MiBhMa, Her}). The notation $\mathcal A^\dagger\{1\}$ denotes the set of all the inner inverses of $A^\dag$.

\begin{propo}\label{propiedadesdiamond}
Let $A$ and $B$ be complex matrices of size $m \times n$. The following conditions are equivalent:
\begin{enumerate}[(a)]
\item[{\rm (a)}] 
$A\,\leq^\diamond\,B$.
\item[{\rm (b)}] 
$A^\dagger\,\leq^-\,B^\dagger$.
\item[{\rm (c)}]
$\mathcal{B}^\dag\{1\}\,\subseteq\,\mathcal{A}^\dag\{1\}$.
\item[{\rm (d)}] 
$A^*\,\leq^\diamond\,B^*$.
\item [{\rm (e)}]
$UAV^*\,\leq^\diamond\,UBV^*$, for any unitary matrices $U$ and $V$ of suitable sizes. 
\end{enumerate}
\end{propo}
\begin{propo}\cite[Theorem 3.4.6]{MiBhMa}\label{caractordenmenosidempotentes}
Let $B \in \Cmn$ be a non-null matrix with a full-rank factorization $B=PQ$ such that $\rm{rk}(B)=b$. Then the class of all matrices $A \in \Cmn$ such that $A\leq^- B$ is given by 
$$\{PTQ :\, T \in {\mathbb C}^{b \times b} \text{ is idempotent}\}.$$
\end{propo}


\begin{propo}\cite[Corollary 1.2.1]{Her}\label{descordenmenosMBM}
Let $A$, $B$ $\in \,\Cmn$ with $a= \rm{rk}(A)$, $b=\rm{rk}(B)$, and $b>a\geq1$. Then $A\,\leq^-\,B$ if and only if there exist unitary matrices $U \in \,\Cmn$ and $V \in \,\Cnn$, a diagonal positive definite matrix $D_a$, and nonsingular matrices $Z \in \C^{(b-a) \times (b-a)}$ and $M \in \,\C^{b \times b}$ such that
\begin{equation}\label{eqdescmenos}
\begin{split}
A= U \left( \begin{array}{ccc}
D_a&&\\
&0_{b-a}&\\
&&0
\end{array}
\right) V^* \qquad \mbox{ and } \qquad 
B= U \left( \begin{array}{cc}
M&\\
&0
\end{array}
\right) V^*,
\end{split}
\end{equation} 
where 
$M= \left( \begin{array}{cc}
D_a&\\
&0
\end{array}
\right)\,\,+\,\,
\left( \begin{array}{c}
-D_a L_1\\
I_{b-a}
\end{array}
\right)Z \left( \begin{array}{cc}
-L_2 D_a & I_{b-a}
\end{array}\right),$ for arbitrary $L_1$ and $L_2$ of appropriate sizes.
\end{propo}

The main aim of this paper is to provide new points of view of diamond and one-sided star partial orders. We would like to highlight that the definition of diamond partial order was not given from an algebraic point of view but by using a geometrical approach by means of range spaces. This paper clarifies some aspects for this partial order which was not investigated before in terms of 
 generalized inverses. In this paper, Theorem \ref{diamondcon1MPMP1} establishes necessary and sufficient conditions characterizing diamond partial order in terms of the orthogonal projectors $AA^\dag$ and $A^\dag A$ as well as by means of 1MP- and MP1-inverses (defined in Section 3) and, somehow, it fills the mentioned gap. Moreover, Theorem \ref{descdiamond} states all explicit forms for matrices $A$ and $B$ ordered under diamond order, which allows us to give examples in a quick way. This investigation leads us to connect, in a very natural way, one-sided star partial orders with 1MP- and MP1-inverses. Likewise for the diamond order, the one-sided star partial order was not originally presented in standard way as a ${\cal G}$-based order (as, for example, minus partial order is presented before), and this paper does it in Theorem \ref{estizq}, and this other gap is filled as well.

The paper is organized as follows. In Section \ref{Sec2}, we give some characterizations and properties of the diamond partial order for rectangular matrices using SVD decompositions and Theorem \ref{Amenos}. In Section \ref{Sec4}, we provide a characterization of the diamond partial order using equalities involving 1MP inverses and its duals. Moreover, characterizations of the one-sided star orders based on 1MP inverses and its duals are given.
\section{Characterizations of the diamond partial order for rectangular matrices}\label{Sec2}

In this section, we obtain a new characterization, based on a SVD of the Moore-Penrose inverse of $A$, 
for $A \in \,\Cmn$ to be below $B \in \,\Cmn$
under the diamond partial order. 
Moreover, by using  Theorem \ref{Amenos}, a characterization and some properties of the diamond partial order are also given. Finally, the predecessors of a matrix $B$ under the diamond  order are obtained from a SVD of $B$. 

\begin{thm}\label{descdiamond}
Let $A$, $B$ $\in \,\Cmn$ with $a={\rm rk}(A)$, $b={\rm rk}(B)$, and $b>a\geq1$. The following statements are equivalent:
\begin{enumerate}[(1)]
\item[{\rm (a)}] 
$A\,\leq^\diamond\,B$.
\item[{\rm (b)}] 
There exist unitary matrices $U \in \Cmm$ and $V \in \Cnn$, a    positive definite diagonal matrix $\Lambda_a$, and nonsingular matrices $Z \in \C^{(b-a) \times (b-a)}$ and $\Gamma\in \C^{b \times b}$ such that 
\begin{equation*}
\begin{split}
A = U{\rm {diag}}(\Lambda_a,0_{b-a},0)V^*, \qquad B=U{\rm {diag}}(\Gamma,0)V^*,
\end{split}
\end{equation*}
with
\begin{equation*}
\Gamma =\left( \begin{array}{cc}
\Lambda_a& L_1\\
L_2 & Z^{-1} + L_2 \Lambda_a^{-1} L_1
\end{array}
\right),
\end{equation*}
 where $L_1$ and $L_2$ are arbitrary matrices of suitable sizes.
\end{enumerate}
\end{thm}
\begin{proof}
$(a)\Rightarrow (b)$
Suposse that $A\,\leq^\diamond\,B$. So, by Proposition \ref{propiedadesdiamond}, $A^\dagger\,\leq^-\,B^\dagger$. It is known that ${\rm rk}(A^\dagger)={\rm rk}(A)=a$ and ${\rm rk}(B^\dagger)={\rm rk}(B)=b$. By applying Proposition \ref{descordenmenosMBM} to matrices $A^\dag$ and $B^\dag$, there exist unitary matrices $V \in {\mathbb C}^{n \times n}$ and $U \in {\mathbb C}^{m \times m}$, a positive definite diagonal matrix $\Sigma_a$, and nonsingular matrices $Z \in \mathbb{C}^{(b-a)\times(b-a)}$ and $M \in \C^{b \times b}$ such that  
\begin{equation*}
A^\dag= V{\rm {diag}}(\Sigma_a,0_{b-a},0)U^*, \qquad B^\dag= V{\rm {diag}}(M,0)U^*
\end{equation*}
and
\begin{equation*}
M= \left( \begin{array}{cc}
\Sigma_a+\Sigma_a L_1 Z L_2 \Sigma_a& -\Sigma_aL_1Z\\
-ZL_2\Sigma_a&Z
\end{array}
\right),
\end{equation*}
where $L_1$ and $L_2$ are arbitrary  matrices. 

So, it is easy to see that the Schur complement of $Z$ in $M$ is $M/Z=\Sigma_a$ (see \cite{PuStIs}). Thus, $M$ is nonsingular and
\begin{equation*}
\begin{split}
M^{-1} =\left( \begin{array}{cc}
\Sigma_a^{-1}& L_1\\
L_2 & Z^{-1} + L_2 \Sigma_a L_1
\end{array}
\right).
\end{split}
\end{equation*}
Thus, by setting $\Gamma:=M^{-1}$ and $\Lambda_a := \Sigma_a^{-1}$, item $(b)$ is proved.

$(b)\Rightarrow (a)$
From $(b)$, we obtain the matrices 
\begin{equation*}
A^\dag= V{\rm {diag}}(\Lambda_a^{-1},0_{b-a},0)U^*, \qquad B^\dag= V{\rm {diag}}(\Gamma^{-1},0)U^*.
\end{equation*}
Let $X$ be an inner inverse of $B^\dag$. So,
\begin{equation*}
\begin{split}
X= U\left( \begin{array}{cc}
\Gamma&R\\
S&T
\end{array}
\right)V^*
\end{split}
\end{equation*} for some complex matrices $R, S, T$ of adequate size. 
Thus, by making some computations using the decompositions given for $A^\dagger, B^\dagger  \text{ and } X$, is easy to see that \begin{equation*}\label{ecmenos}
\begin{split}
B^\dagger X  A^\dagger = A^\dagger X B^\dagger = A^\dagger.
\end{split}
\end{equation*} 
Moreover,
\begin{equation*} 
\begin{split}
A^\dagger \! X \! A^\dagger&= \!
V  
\left(\!\! \begin{array}{cc}
\left(\!\! \begin{array}{cc} \Lambda_a^{-1}&0\\0&0_{b-a}\end{array}
\!\!\right) &0\\
0&0
\end{array}
\!\!\right) U^*U
\left(\!\! \begin{array}{cc}
\Gamma&R\\
S&T
\end{array}
\!\!\right) V^* A^\dagger 
\\
&=\! V \left(\!\! \begin{array}{cc}
\left(\!\! \begin{array}{cc} \Lambda_a^{-1}&0\\0&0_{b-a}\end{array}\!\!
\right)\!\Gamma&\left(\!\!\begin{array}{cc} \Lambda_a^{-1}&0\\0&0_{b-a}\end{array}
\!\!\right)\!R\\
0&0
\end{array}
\!\!\right)\left(\!\! \begin{array}{cc}
\left(\!\! \begin{array}{cc} \Lambda_a^{-1}&0\\0&0_{b-a}\!\end{array}
\!\!\right) &0\\
0&0
\end{array}
\!\!\right) U^*\\
&=\!V \left(\!\! \begin{array}{cc}
\left(\!\! \begin{array}{cc} \Lambda_a^{-1}&0\\0&0_{b-a}\end{array}
\!\!\right)\!\Gamma\!\left(\!\! \begin{array}{cc} \Lambda_a^{-1}&0\\ 0&0_{b-a}\end{array}
\!\!\right) &0\\
0&0
\end{array}
\!\!\right)U^*.
\end{split}
\end{equation*}
Note that
\begin{equation*}
\begin{split}
\left( \begin{array}{cc} \Lambda_a^{-1}&0\\ 0&0_{b-a}\end{array}
\right)\Gamma&=\left( \begin{array}{cc} \Lambda_a^{-1}&0\\ 0&0_{b-a}\end{array}
\right)\left( \begin{array}{cc} \Lambda_a&L_1\\L_2&Z^{-1} + L_2\Lambda_a^{-1}L_1\end{array}
\right)\\
&=\left( \begin{array}{cc} I_a&\Lambda_a^{-1}L_1\\0&0\end{array}
\right). 
\end{split}
\end{equation*}
So,
\begin{equation*}
\begin{split}
\left( \begin{array}{cc} \Lambda_a^{-1}&0\\ 0&0_{b-a}\end{array}
\right)\Gamma\left( \begin{array}{cc} \Lambda_a^{-1}&0\\ 0&0_{b-a}\end{array}
\right)&=\left( \begin{array}{cc} I_a&\Lambda_a^{-1}L_1\\0&0\end{array}
\right)\left( \begin{array}{cc} \Lambda_a^{-1}&0\\ 0&0_{b-a}\end{array}
\right)\\
&=\left( \begin{array}{cc} \Lambda_a^{-1}& 0\\ 0&0_{b-a}\end{array}
\right).
\end{split}
\end{equation*}
Thus, $A^\dagger X A^\dagger=A^\dagger$, i.e., $X$ is an inner inverse of $A^\dagger$. From Proposition \ref{propiedadesdiamond} (c), $A\,\leq^\diamond\,B$.
\end{proof}
Notice that if $a = {\rm rk}(A) = b = {\rm rk}(B)$, the proof of Theorem \ref{descdiamond} is immediate.

\begin{rem}{\rm
If $A$ and $B$ are represented by the decompositions given in Theorem \ref{descdiamond} $(b)$, then 
\begin{equation*}
\begin{split}
B^\dagger = V \left( \begin{array}{ccc}
\Lambda_a^{-1}+\Lambda_a^{-1}L_1ZL_2\Lambda_a^{-1}& -\Lambda_a^{-1}L_1Z&0\\
-ZL_2\Lambda_a^{-1} & Z&0\\
0&0&0
\end{array}
\right) U^*.
\end{split}
\end{equation*}}
\end{rem}
In what follows, another characterizations of the diamond order are given.
\begin{thm}\label{caractdiamondA+}
Let $A$, $B$ $\in \,\Cmn$. The following statements are equivalent:
\begin{enumerate}[(1)]
\item[{\rm (a)}] 
$A\,\leq^\diamond\,B$.
\item[{\rm (b)}] 
$A^\dagger\,=\,A^\dagger B A^\dagger\,+\,A^\dagger W A^\dagger \,-\,A^\dagger B B^\dagger W B^\dagger B A^\dagger$, for  arbitrary $W\,\in\,\mathbb{C}^{m\times n}$.
\end{enumerate}
\end{thm}
\begin{proof}
Suppose $A\,\leq^\diamond\,B$. From Proposition \ref{propiedadesdiamond}, $\mathcal B^\dagger\{1\}\,\subseteq\,\mathcal A^\dagger\{1\}$. Let $T\,\in\,\mathcal B^\dagger\{1\}$. Since $B \, \in \mathcal B^\dagger\{1\}$, by Theorem \ref{Amenos}, $T\,=\, B + W - B B^\dagger W B^\dagger B$, for some 
$W\,\in\,\mathbb{C}^{m\times n}$.  By hypothesis, $T\,\in \,\mathcal A^\dagger\{1\}$. By replacing the expression of $T$ in equality $A^\dagger T A^\dagger \,=\,A^\dagger$ we obtain
$$A^\dagger\,=\,A^\dagger(B + W - B B^\dagger W B^\dagger B)A^\dagger\,=\,A^\dagger B A^\dagger\,+\,A^\dagger W A^\dagger \,-\,A^\dagger B B^\dagger W B^\dagger B A^\dagger.$$ Note that the arbitrariness of $T$ implies that $W\,\in\,\Cmn$ is arbitrary as well.

Suppose now that $A^\dagger\,=\,A^\dagger B A^\dagger\,+\,A^\dagger W A^\dagger \,-\,A^\dagger B B^\dagger W B^\dagger B A^\dagger$ holds for arbitrary  $W\,\in\,\mathbb{C}^{m\times n}$. Let $G\,\in\,\mathcal B^\dagger\{1\}$. By Theorem \ref{Amenos}, $G\,=\, B + X - B B^\dagger X B^\dagger B$ for some $X \,\in\,\mathbb{C}^{m\times n}$. Hence, by using  the arbitrariness of $W$ in the hypothesis,
$$
A^\dagger G A^\dagger \,=\,A^\dagger(B + X - B B^\dagger X B^\dagger B)A^\dagger\,=\,A^\dagger B A^\dagger\,+\,A^\dagger X A^\dagger \,-\,A^\dagger B B^\dagger X B^\dagger B A^\dagger\,=\,A^\dagger.
$$ 
Thus, $G\,\in\,\mathcal A^\dagger\{1\}$. Hence, $A\,\leq^\diamond\,B$ by Proposition \ref{propiedadesdiamond}.
\end{proof}

We can derive some identities by considering different particular matrices $W$ in the previous theorem.
\begin{propo}\label{propMP}
Let $A$, $B$ $\in \,\Cmn$ such that $A\,\leq^\diamond\,B$. Then
\begin{enumerate}[(a)]
\item[{\rm (a)}] 
$A^\dagger\,=\,A^\dagger B A^\dagger$.
\item [{\rm (b)}] 
$A^\dagger\,=\,A^\dagger (A B^\dagger B) A^\dagger$. Thus, $A^\dagger A\,=\,A^\dagger A B^\dagger B A^\dagger A$ (i.e., $B^\dagger B$ is a \{1\}-inverse of $A^\dagger A$). Also, $A A^\dagger\,=\,A B^\dagger B A^\dagger$. 
\item [{\rm (c)}] 
$A^\dagger\,=\,A^\dagger (B B^\dagger A B^\dagger B) A^\dagger$.
\item [{\rm (d)}] 
$(A^\dagger)^*\,=\,(A^\dagger)^* B^\dagger B A^* B B^\dagger (A^\dagger)^*$.
\item [{\rm (e)}] 
$A^\dagger\,=\,A^\dagger (B B^\dagger A) A^\dagger$. Thus, $A A^\dagger \,=\,A A^\dagger  B B^\dagger A A^\dagger$ (i.e., $B B^\dagger$ is a \{1\}-inverse of $A A^\dagger$). Also $A^\dagger A \,=\,A^\dagger B B^\dagger A $.
\end{enumerate}
\end{propo}
\begin{proof}
Item (a) follows by setting $W\,=\,0$ in Theorem \ref{caractdiamondA+}. By using this expression in (a) and particularizing $W\,=\,B A^\dagger A$, $W\,=\,A$, and $W\,=\,A A^\dagger B$, respectively, and by making some computations we get (b), (c), and (e). Item (d) follows immediatelly from (c).
\end{proof}

Next, we give two characterizations for the diamond partial order. They do not use geometrical approaches but only algebraic equations.
\begin{thm}\label{caractdiamond}
Let $A$, $B$ $\in \,\Cmn$. The following statements are equivalent:
\begin{enumerate}[(a)]
\item[{\rm (a)}]
$A\,\leq^\diamond\,B$.
\item[{\rm (b)}]
$A^\dagger\,=\,A^\dagger B A^\dagger$ \qquad and \qquad
$A\,=\,B B^\dagger A\,=\,A B^\dagger B$.
\item[{\rm (c)}]
$A^\dagger\,=\,A^\dagger B A^\dagger\,=\,A^\dagger B B^\dagger\,=\,B^\dagger B A^\dagger$.
\end{enumerate}
\end{thm}
\begin{proof}
(a) $\Rightarrow$ (b) Since $\mathcal{R}(A)\,\subseteq\,\mathcal{R}(B)$, there exists a matrix $X$ such that $A\,=\,BX$. Premultiplying by $BB^\dagger$ we obtain $BB^\dagger A\,=\,BB^\dagger BX\,=\,BX\,=\,A$. Analogously, from $\mathcal{R}(A^*)\,\subseteq\,\mathcal{R}(B^*)$ we obtain $AB^\dagger B\,=\,A$. The equality $A^\dagger\,=\,A^\dagger B A^\dagger$ was proved in Proposition \ref{propMP}(a).

(b) $\Rightarrow$ (c)  By hypothesis, we have $A=A B^\dag B$. Then
$A^*A(B^\dag B)(B^\dag B)^* = A^* A B^\dag B = A^* A$, 
which is a hermitian matrix. Hence, by Proposition \ref{propA+0} (\ref{MP_distributiva}), $(A(B^\dag B))^\dag = (B^\dag B)^\dag A^\dag.$ Analogously, from $A=B B^\dag A$, we obtain $((BB^\dag)A)^\dag= A^\dag(B B^\dag)^\dag.$ Thus, by Proposition \ref{propA+0} (\ref{A_proj_orto_MP}),
$A^\dag=((B B^\dag) A)^\dag = A^\dag (B B^\dag)^\dag = A^\dag B B^\dag$ and $A^\dag=(A (B^\dag B))^\dag = (B^\dag B)^\dag A^\dag  = B^\dag B A^\dag.$

(c) $\Rightarrow$ (b) From $A^\dag=A^\dag B B^\dag$, we have that $(A^\dag)^*A^\dag(BB^\dag )(BB^\dag)^* = (A^\dag)^* A^\dag (B B^\dag)= (A^\dag)^* A^\dag$, which is a hermitian matrix. By Proposition \ref{propA+0}, 
$A=(A^\dag(BB^\dag))^\dag = (B B^\dag)^\dag A = B B^\dag A.$
Analogously,  we get $((B^\dag B)A^\dag)^\dag= A B^\dag B$ from $A^\dag=B^\dag B A^\dag$. So, 
$A = (A^\dag(B B^\dag))^\dag = B B^\dag A$ and
$A = ((B^\dag B)A^\dag)^\dag = A B^\dag B.$

(b) $\Rightarrow$ (a) From $A=B(B^\dag A)$ and $A=(AB^\dag)B$ we have 
$\mathcal{R}(A)  \subseteq \mathcal{R}(B)$ and $\mathcal{R}(A^*) \subseteq \mathcal{R}(B^*).$ 
Using $A^\dag=A^\dag B A^\dag$ and Proposition \ref{propA+0} we obtain 
$$
A B^* A = ((A B^* A)^*)^* = (A^* A (A^\dag B A^\dag) A A^*)^* = (A^* A A^\dag A A^*)^* = A A^* A,
$$
which completes the proof.
\end{proof}

Theorem \ref{descdiamond} allows us to characterize the successors of a matrix $A$ under the diamond order from a SVD of $A^\dag$. 
In what follows, we obtain the predecessors of a matrix $B$ under the diamond order from a SVD of $B$. 
\begin{thm}\label{predecdiamond}
Let $B \in \mathbb C^{m \times n}$ written as
\begin{equation}\label{SVD_B}
B= U\left( \begin{array}{cc}
D_b&0\\
0&0
\end{array}
\right) V^*,
\end{equation} where $U \in\mathbb{C}^{m \times m}$ and $V\in\mathbb{C}^{n \times n}$ are unitary matrices and $D_b \in \mathbb C^{b \times b}$ is a positive definite diagonal matrix. The following conditions are equivalent:
\begin{enumerate}[(a)]
\item[{\rm (a)}] There exists $A \in {\mathbb C}^{m \times n}$ such that $A\,\leq^\diamond\,B$.
\item[{\rm (b)}] The expression for $A \in {\mathbb C}^{m \times n}$ is
\[
A= U \left( \begin{array}{cc}
T&0\\
0&0
\end{array}
\right) V^*
\] with $T \in \C^{b \times b}$ and $T \,\leq^\diamond\,D_b$.
\end{enumerate} 
\end{thm}
\begin{proof}
(a) $\Rightarrow$ (b) Suppose that $A\,\leq^\diamond\,B$. Let $
A= U \left( \begin{array}{cc}
T&E\\
F&G
\end{array}
\right) V^*
$ be partitioned accordingly to the sizes of the blocks of $B$. 
From Proposition \ref{propiedadesdiamond}, we have
$$\left( \begin{array}{cc}
T&E\\
F&G
\end{array}
\right)\,\leq^\diamond \,\left( \begin{array}{cc}
D_b&0\\
0&0
\end{array}
\right).$$ 

Since $\mathcal{R}(A)\,\subseteq\,\mathcal{R}(B)$, there exists a matrix $X$ such that $A\,=\,BX$. Then
$$
U \left( \begin{array}{cc}
T&E\\
F&G
\end{array}
\right) V^*\,=\,U \left( \begin{array}{cc}
D_b&0\\
0&0
\end{array}
\right) V^*X,
$$ 
from which
$$
\left( \begin{array}{cc}
T&E\\
F&G
\end{array}
\right)\,=\,\left( \begin{array}{cc}
D_b&0\\
0&0
\end{array}
\right) V^*X V.
$$

Let $V^*X V\,=\,\left( \begin{array}{cc}
R_1&R_2\\
R_3&R_4
\end{array}
\right)$. Then,
$$
\left( \begin{array}{cc}
T&E\\
F&G
\end{array}
\right)\,=\,\left( \begin{array}{cc}
D_b&0\\
0&0
\end{array}
\right) \left( \begin{array}{cc}
R_1&R_2\\
R_3&R_4
\end{array}
\right)\,=\,\left( \begin{array}{cc}
D_bR_1&D_bR_2\\
0&0
\end{array}
\right).
$$ 
Hence, $F=0$ and $G=0$. Similary, from $\mathcal{R}(A^*)\,\subseteq\,\mathcal{R}(B^*)$, we obtain $E=0$. 

Thus, $A= V \left( \begin{array}{cc}
T&0\\
0&0
\end{array}
\right) U^*$ and
$\left( \begin{array}{cc}
T&0\\
0&0
\end{array}
\right)\,\leq^\diamond \,\left( \begin{array}{cc}
D_b&0\\
0&0
\end{array}
\right)$. By definition, $T\,\leq^\diamond\,D_b$.

(b) $\Rightarrow$ (a)
Since $T\,\leq^\diamond\,D_b$, an easy computation shows that $\left( \begin{array}{cc}
T&0\\
0&0
\end{array}
\right)\,\leq^\diamond \,\left( \begin{array}{cc}
D_b&0\\
0&0
\end{array}
\right)$. From Proposition \ref{propiedadesdiamond}, $A\,\leq^\diamond\,B$.
\end{proof}

\begin{rem}{\rm
Under the assumptions and notations of Theorem \ref{predecdiamond}, it is easy to see that the following statements are true:
\begin{enumerate}[(a)]
\item[{\rm (a)}] 
If $T\,\leq^\diamond \,D_b$ and $T$ is nonsingular, then $T\,=\,D_b$.
\item[{\rm (b)}] 
$\left\{A\,\in \,\mathbb C^{m \times n}: A\,\leq^\diamond\,B\right\}= \left\{  U \left( \begin{array}{cc}
(XD_b^{-1})^\dag&0\\
0&0
\end{array}
\right) V^*\,:\,X^2=X, X \in \,\mathbb C^{b \times b}\right\}.$ In fact, from $T \,\leq^\diamond\, D_b$, by Proposition \ref{propiedadesdiamond}, we have that 
$T^\dag \,\leq^-\, D_b^\dag=D_b^{-1}$.
Since $D_b^{-1}=I_bD_b^{-1}$ is a trivial full-rank factorization of $D_b^{-1}$, by Proposition \ref{caractordenmenosidempotentes}, there exists an idempotent matrix $X \in \,\mathbb C^{b \times b}$ such that $T^\dagger\,=I_b\,X D_b^{-1}$, from where $T\,=\,(X D_b^{-1})^\dag$ with $X^2=X$. 
\end{enumerate}}
\end{rem}
\section{Partial orders defined from 1MP- and MP1-inverses}\label{Sec4}

Recently, the authors introduced in \cite{HeLaTh01} a new class of generalized inverses of rectangular matrices, the 1MP- and MP1-inverses. 
An extension to rings and applications can be found in \cite{MaMoCr,RaLj}.
For each $A^- \in \A\{1\}$, it was defined 
$$
A^{-\dagger}:=A^- A A^\dagger \in \mathbb C^{n \times m} \qquad \text{ and } \qquad A^{\dagger -}:= A^\dagger A A^-  \in \mathbb C^{n \times m}.
$$

The symbols $\mathcal{A}{\{-\dagger \}}$ and $\mathcal{A}{\{\dagger - \}}$ stands for the set of all 1MP- and MP1-inverses of $A$, respectively. Moreover,
$$\A{\{-\dagger\}}=\{A^-AA^\dagger:  A^- \in \A\{1\}\}\quad \text{ and }\quad \A{\{\dagger -\}}=\{A^\dagger A A^-:  A^- \in \A\{1\}\}.$$

The following result provides some characterizations of the diamond order in terms of the orthogonal projectors $AA^\dag$ and $A^\dag A$ and of the 1MP- and MP1-inverses.

\begin{thm}\label{diamondcon1MPMP1}
Let  $A, B \in \,\mathbb C^{m \times n}$. The following conditions are equivalent:
\begin{enumerate}[(a)]
\item[{\rm (a)}]
$A\,\leq^\diamond \,B$.
\item [{\rm{(b)}}]
$B^\dagger B A^\dagger A \,=\,A^\dagger B A^\dagger A \,=\,A^\dagger A \quad \text{and} \quad A A^\dagger B B^\dagger \,=\, A A^\dagger B A^\dagger \,=\, A A^\dagger$.
\item [{\rm{(c)}}]
$B^\dagger B A^{\dagger -} \,=\,A^\dagger B A^{\dagger -}\,=\,A^{\dagger -} \quad \text{and} \quad A^{- \dagger} B B^\dagger\,=\,A^{- \dagger} B A^\dagger\,=\,A^{- \dagger}$ for some $A^- \in \mathcal A\{1\}$.
\item [{\rm{(d)}}]
$B^\dagger B A^{\dagger -} \,=\,A^\dagger B A^{\dagger -}\,=\,A^{\dagger -} \quad \text{and} \quad A^{- \dagger} B B^\dagger\,=\,A^{- \dagger} B A^\dagger\,=\,A^{- \dagger}$ for all $A^- \in \mathcal A\{1\}$.
\end{enumerate}
\end{thm}
\begin{proof} 
(a) $\Leftrightarrow$ (b) It is an immediate consequence of Theorem \ref{caractdiamond}. The other equivalences follow from the equalities $A^\dag A = A^{\dag -}A$ and $A A^\dag = A A^{-\dag}$, for any $A^- \in \A\{1\}$.
\end{proof}

The star order is defined for complex rectangular matrices (see \cite{MiBhMa}). Given $A, B \in \Cmn$, it is said that:
$$A\,\leq^*\,B\, (A \text{ is below } B \text{ under the star order) if } AA^*\,=\,BA^* \quad \mbox{and}\quad A^*A\,=\,A^*B.$$ If only one of the two conditions of definition is considered, other two matrix orders are obtained by adding some conditions. It is well known that:
\begin{itemize}
\item $A\,*\leq \,B$ ($A$ is below $B$ under left star order) if $A^*A\,=\,A^*B \quad \mbox{and} \quad {\Ra}(A)\,\subseteq \,{\Ra}(B)$;
\item $A\, \leq* \,B$ ($A$ is below $B$ under right star order) if $AA^*\,=\,BA^* \quad \mbox{and} \quad {\Ra}(A^*)\,\subseteq\,{\Ra}(B^*)$. 
\end{itemize}
Also, it is known that $A^*$ can be replaced by $A^\dagger$ in the equalities involved in the previous definitions. Some related results for these and other one-sided partial orders can be found in \cite{CvMoWe,DoMa,ZhPa}.  

The one-sided star orders can be considered as ${\cal G}$-based partial orders by considering the adequate generalized inverses, namely, 1MP- and MP1-inverses.
\begin{thm} \label{estizq} Let $A, B\in \Cmn$. Then 
\begin{enumerate}
\item [{\rm (a)}] $A\,*\leq\,B \qquad$ iff $\qquad A^{- \dag}A=A^{- \dag}B \quad \text{and} \quad AA^{- \dag}=BA^{- \dag}$, for some $A^{- \dag}\in \mathcal{A}{\{-\dagger \}}$. 
\item [{\rm (b)}] $A\,\leq*\,B \qquad$ iff $\qquad A^{\dag -}A=A^{\dag -}B \quad \text{and} \quad AA^{\dag -}=BA^{\dag -}$, for some $A^{\dag-}\in \mathcal{A}{\{\dagger -\}}$.
\end{enumerate}
\end{thm}
\begin{proof}
$(a)$ Suppose that $A\,*\leq\,B$. So, $A^\dag A = A^\dag B$. Pre-multiplying for $A$ and taking into account that  $A^\dag$ is a inner inverse of $A$, we obtain 
\begin{equation} \label{eq_no_esta}
A= A A^\dag B.
\end{equation}  
It is well known that the left star order implies the minus order, so there exists a matrix $A^{-} \in \A\{1\}$ such that $A^- A=A^- B$ 
and $A A^-=B A^-$. Pre-multiplying $(\ref{eq_no_esta})$ by this  matrix $A^{-}$, we obtain $A^{- \dag}A=A^{- \dag}B$ since $A^{-} A= A^{- \dag}A$. Now, post-multiplying $A A^-=B A^-$   by $AA^\dag$ we obtain $AA^-AA^\dag= B A^- AA^\dag$.
So, $AA^{- \dag}=BA^{- \dag}$. 

Now, suppose that exists $A^{- \dag}\in \mathcal{A}{\{-\dagger \}}$  such that $A^{- \dag}A=A^{- \dag}B \quad \text{and} \quad AA^{- \dag}=BA^{- \dag}.$

Pre-multiplyng $A^{- \dag}A=A^{- \dag}B$ by $ A^\dag A$, it results that $A^\dag A= A^\dag B$.
Similary, post-multiplying $AA^{- \dag}=BA^{- \dag}$ by $A$, we obtain $A A^- A= B A^- A$. So, $A= B (A^- A)$, i.e, ${\Ra}(A)\, \subseteq \,{\Ra}(B)$. Therefore $A\,*\leq\,B$.

$(b)$ The proof is similar to $(a)$.
\end{proof}

\section{Acknowledgement}
The authors would like to thank the anonymous referees for their careful reading of this article and their valuable comments and suggestions that help us to improve the reading of the paper.


\end{document}